\documentclass[reqno,a4paper,11pt]{article}

\usepackage{amsmath,amssymb,amsthm}

\usepackage{mathrsfs}

\usepackage{mathtools}
\usepackage{commath}

\usepackage{thmtools}
\usepackage{thm-restate}

\usepackage{cases}
\usepackage{enumitem}
\setlist[enumerate,1]{label=(\roman*)}

\usepackage{xcolor}
\usepackage[pdftex, pdfborderstyle={/S/U/W 0}]{hyperref}
\hypersetup{
    colorlinks=true,
    linkcolor=red,
    citecolor=blue,
}

\usepackage{cleveref}

\usepackage{etoolbox}

\usepackage{comment}

\usepackage[numbers]{natbib}

\numberwithin{equation}{section}

\ifdefined\thmcolor
\declaretheoremstyle[
  shaded={bgcolor=\thmcolor}
]{plain}
\else
\fi

\ifdefined\defcolor
\declaretheoremstyle[
  headfont=\normalfont\bfseries,
  bodyfont=\normalfont,
  shaded={bgcolor=\defcolor}
]{noital}
\else
\declaretheoremstyle[
  headfont=\normalfont\bfseries,
  bodyfont=\normalfont,
]{noital}
\fi

\declaretheorem[style=plain,numberwithin=section,name=Theorem]{theorem}

\declaretheorem[style=plain,sibling=theorem,name=Lemma]{lemma}

\declaretheorem[style=plain,sibling=theorem,name=Conjecture]{conjecture}

\declaretheorem[style=plain,sibling=theorem,name=Question]{question}
\declaretheorem[style=plain,sibling=theorem,name=Observation]{observation}

\declaretheorem[style=plain,numbered=no,name=Theorem]{theorem-n}
\declaretheorem[style=plain,numbered=no,name=Proposition]{proposition-n}
\declaretheorem[style=plain,numbered=no,name=Lemma]{lemma-n}
\declaretheorem[style=plain,numbered=no,name=Corollary]{corollary-n}
\declaretheorem[style=plain,numbered=no,name=Conjecture]{conjecture-n}
\declaretheorem[style=plain,numbered=no,name=Claim]{claim-n}
\declaretheorem[style=plain,numbered=no,name=Fact]{fact-n}
\declaretheorem[style=plain,numbered=no,name=Open Problem]{openproblem-n}
\declaretheorem[style=plain,numbered=no,name=Question]{question-n}
\declaretheorem[style=plain,numbered=no,name=Observation]{observation-n}

\declaretheorem[style=noital,sibling=theorem,name=Definition]{definition}

\declaretheorem[style=noital,numbered=no,name=Remark]{remark-n}
\declaretheorem[style=noital,numbered=no,name=Definition]{definition-n}
\declaretheorem[style=noital,numbered=no,name=Construction]{construction-n}
\declaretheorem[style=noital,numbered=no,name=Example]{example-n}

\newcommand{\defined}{\mathrel{\coloneqq}}

\DeclarePairedDelimiter{\p}{\lparen}{\rparen}

\newcommand{\st}{\mathbin{\colon}}

\undef{\set}
\DeclarePairedDelimiter{\set}{\lbrace}{\rbrace}

\undef{\emptyset}
\newcommand{\emptyset}{\varnothing}

\DeclarePairedDelimiter{\card}{\lvert}{\rvert}

\newcommand{\union}{\mathbin{\cup}}
\newcommand{\inter}{\mathbin{\cap}}

\newcommand{\from}{\colon}

\DeclarePairedDelimiter{\floor}{\lfloor}{\rfloor}

\newcommand{\setm}[1]{\setminus\set{#1}}

\undef{\mod}
\newcommand{\mod}[1]{\ (\mathrm{mod}\ #1)}

\undef{\abs}
\DeclarePairedDelimiterX{\abs}[1]
  {\lvert}{\rvert}{\ifblank{#1}{\,\cdot\,}{#1}}

\undef{\norm}
\DeclarePairedDelimiterX{\norm}[1]
  {\lVert}{\rVert}{\ifblank{#1}{\,\cdot\,}{#1}}

\DeclarePairedDelimiterX{\inner}[2]
  {\langle}{\rangle}{\ifblank{#1}{\,\cdot\,}{#1},\ifblank{#2}{\,\cdot\,}{#2}}

\DeclarePairedDelimiterX{\absinner}[2]
  {|\langle}{\rangle|}{\ifblank{#1}{\,\cdot\,}{#1},\ifblank{#2}{\,\cdot\,}{#2}}

\DeclareMathDelimiter{\given}
  {\mathbin}{symbols}{"6A}{largesymbols}{"0C}

\DeclareMathOperator{\Prob}{\mathbb{P}}
\DeclarePairedDelimiterXPP{\prob}[1]
  {\Prob}{\lparen}{\rparen}{}
  {\renewcommand{\given}{\nonscript\;\delimsize\vert\nonscript\;\mathopen{}}#1}

\DeclareMathOperator{\Expec}{\mathbb{E}}
\DeclarePairedDelimiterXPP{\expec}[1]
  {\Expec}{\lparen}{\rparen}{}
  {\renewcommand{\given}{\nonscript\;\delimsize\vert\nonscript\;\mathopen{}}#1}

\DeclareMathOperator{\Var}{Var}
\DeclarePairedDelimiterXPP{\var}[1]
  {\Var}{\lparen}{\rparen}{}
  {\renewcommand{\given}{\nonscript\;\delimsize\vert\nonscript\;\mathopen{}}#1}

\DeclareMathOperator{\Cov}{Cov}
\DeclarePairedDelimiterXPP{\cov}[2]
  {\Cov}{\lparen}{\rparen}{}{#1,#2}

\newcommand{\eps}{\varepsilon}
\newcommand{\sseq}{\subseteq}

\let\l\relax
\newcommand{\l}{\ell}

\let\SS\relax

\newcommand{\RR}{\mathbb{R}}
\newcommand{\SS}{\mathbb{S}}

\newcommand{\ZZ}{\mathbb{Z}}

\newcommand{\cZ}{\mathcal{Z}}

\usepackage{geometry}
\usepackage{scrextend}

\usepackage[T1]{fontenc}
\usepackage{titlesec}

\titleformat{\section}{\centering\bfseries\scshape\Large}{\thesection}{1em}{}
\titleformat{\subsection}{\bfseries\scshape\large}{\thesubsection}{1em}{}

\newcommand{\lhc}[1]{}

\newcommand{\agnijoc}[1]{}

\newcommand{\vol}[1]{\text{Vol}(#1)}
\newcommand{\swap}{\leftrightarrow}

\begin{document}






\title{\textsc{\bfseries Uniformly balanced $H$-factors in multicoloured complete graphs}}

\renewcommand{\thefootnote}{\fnsymbol{footnote}}

\author{\textsc{Agnijo Banerjee}\footnotemark[1] \and \textsc{Lawrence Hollom}\footnotemark[1]}

\footnotetext[1]{\href{mailto:ab2558@cam.ac.uk}{ab2558@cam.ac.uk} and \href{mailto:lh569@cam.ac.uk}{lh569@cam.ac.uk} respectively. Department of Pure Mathematics and Mathematical Statistics (DPMMS), University of Cambridge, Wilberforce Road, Cambridge, CB3 0WA, United Kingdom}

\renewcommand{\thefootnote}{\arabic{footnote}}

\maketitle

\begin{abstract}
    A \emph{balanced colouring} of a graph is one in which every colour appears the same number of times.
    Given a fixed graph $H$ on $r$ vertices and a balanced $k$-colouring of the complete graph $K_{nrk}$, Hollom (2025) asked the following question: can we always find an $H$-factor $F$ covering all vertices of the complete graph $K_{nrk}$ such that the inherited colouring of $F$ is almost balanced?
    This is known to be the case for palettes of only two colours, or when $H$ is only a single edge.
    We answer the above question in full, finding an $H$-factor which is at most $C_{r,k}$ edges away from being balanced, where $C_{r,k}$ depends only on $r$ and $k$.
    In fact, we work in the more general setting wherein our palette of colours is a subset of $\SS^{d-1}$, and find an $H$-factor where the sum of the colours of all edges has bounded Euclidean norm.
\end{abstract}


\section{Introduction}
\label{sec:intro}

Given a colouring of the complete graph $K_n$, questions concerning when we can find specific subgraphs $F \sseq K_n$ whose colourings have particular properties are found throughout combinatorics.
For example, Ramsey theory asks for which graphs $F$ one can guarantee a monochromatic copy of $F$ in $K_n$.
More generally, discrepancy theory asks merely for a copy of $F$ with one colour occurring as frequently as possible.
Conversely, one may also ask for balanced structures, wherein each colour should occur in $F$ (approximately) the same number of times. Indeed, we call an edge-colouring of a graph \emph{balanced} if each colour occurs exactly the same number of times.

We work in the latter setting, letting $F$ be an $H$-factor for some fixed graph $H$ (that is, a graph consisting of vertex-disjoint copies of $H$) and answer a question of Hollom \cite{Hol24}, proving the following theorem.

\begin{theorem}
\label{thm:main-k-colouring}
    For all integers $k,r\geq 2$ there is a constant $C=C_{k,r}$ such that the following holds.
    For every graph $H$ on $r$ vertices, integer $n$ and balanced $k$-colouring $c\from E(K_{nrk}) \to [k]$, there is an $H$-factor $F$ of $K_{nrk}$ such that
    \begin{equation*}
        \sum_{i=1}^k \abs[\bigg]{\card{E(F)\inter c^{-1}(i)} - \frac{\card{E(F)}}{k}} \leq C,
    \end{equation*}
    where $C$ depends only on $k$ and $r$. In particular, we may take
    \[C = (8kr)^{(8kr)^k}.\]
\end{theorem}

We remark here that, while we have removed any $n$-dependence from $C$, it seems highly likely that the true value is far smaller than that given above. 
Indeed, as will be discussed later, it has been conjectured that the correct value of $C$ is $O(k^2)$ when $H$ is a single edge.
However, in general, the authors are not aware of any constructions which show that $C$ must grow with $k$ or $r$.


In fact, we will prove \Cref{thm:main-k-colouring} via a more general result, in which the colours are unit vectors, as follows. 
In the following theorem, and throughout this paper, $\norm{\cdot}$ is the Euclidean norm on $\RR^d$.

\begin{theorem}
\label{thm:main}
    Fix integers $k,r,d\geq 2$, a graph $H$ on $r$ vertices, and a subset $S\sseq \SS^{d-1} \sseq \RR^d$ of size $k$.
    Then there is a constant $C=C_{S,r}$ such that the following holds.
    For every integer $n$, positive real number $\alpha > 0$, and colouring $c\from E(K_{nr}) \to S$ such that $\norm{\sum_{e\in E(K_{nr})} c(e)} \leq \alpha$, there is an $H$-factor $F$ of $K_{nr}$ such that
    \begin{equation*}
        \norm[\bigg]{\sum_{e\in E(F)} c(e) \; } \leq C + \frac{2\alpha}{n}.
    \end{equation*}
    Moreover, if $S$ is the vertex set of a regular $(d+1)$-simplex in $\SS^{d-1}$, then we may take
    \begin{equation*}
        C = (8dr)^{(8dr)^{d+1}}.
    \end{equation*}
\end{theorem}

We discuss the history of the problem of finding balanced subgraphs in \Cref{subsec:history} and then give an outline of our proof in \Cref{subsec:outline}.
In \Cref{sec:deduction}, we show how \Cref{thm:main} implies \Cref{thm:main-k-colouring}, and then we give the full details of the proof of \Cref{thm:main} in \Cref{sec:clique-factor,sec:simplex-bound,sec:h-factor}.
Finally, in \Cref{sec:conclusion}, we give some concluding remarks and suggest directions for future research. 

\subsection{History and known results}
\label{subsec:history}

There have been several different lines of research which can be viewed in terms of finding balanced subgraphs of coloured graphs.
For example, after a series of partial results, F\"{u}redi and Kleitman \cite{FK92} showed that, if the edges of $K_n$ are coloured from the palette $\set{1,\dotsc,n-1}$, then there is a spanning tree with edges having colours summing to $0\mod n$.
This result was later generalised by Schrijver and Seymour \cite{SS91}.

In a similar direction, Caro and Yuster \cite{CY16} showed that every balanced colouring of $K_n$ with colours $\set{-r,\dotsc,r}$ contains almost zero-sum copies of a fixed spanning graph with bounded maximum degree. 
This was then extended by Caro, Hansberg, Lauri, and Zarb \cite{CHLZ20}, who worked on the following problem of, given a 2-colouring $f\from E(G) \to \set{-1,1}$, what conditions on $\abs{f(G)}$ guarantee a zero-sum spanning tree (for various graphs $G$, amongst other results).
Moreover, they asked whether every 2-edge-coloured complete graph of order $4n$ with equally many edges of each colour contains a perfect matching which also contains equally many edges of each colour.

This question was answered in the affirmative by Ehard, Mohr, and Rautenbach \cite{EMR20}, and by Kittipassorn and Sinsap \cite{KS23}, who proved that such a perfect matching must always exist. The latter also asked whether this result on perfect matchings can be generalised to many colours.

This question was soon answered by means of a counterexample, as Pardey and Rautenbach \cite{PR22}, amongst other results, exhibited a 3-colouring of the edges of $K_6$ with no perfectly balanced perfect matching (i.e.\ with two edges in each colour). This naturally leads to the question of finding a perfect matching which is merely as close to balanced as possible.
To this end, define the \emph{deviation} of a perfect matching $M\sseq K_{2kn}$ with respect to a balanced colouring $c\from E(K_{2kn}) \to [k]$ to be
\begin{equation}
\label{eq:matching-deviation}
    f_c(M) \defined \sum_{i=1}^k \abs[\big]{\card{c^{-1}(i) \inter M} - n}.
\end{equation}
The quantity $f_c(M)$ can be seen to measure how far the colouring of $M$ is from being balanced.
Pardey and Rautenbach \cite{PR22} thus relaxed the question of \cite{KS23} to the following conjecture.
\begin{conjecture}
\label{conj:balanced-perfect-matching}
    If $n$ and $k$ are positive integers, and $c\from E(K_{2kn}) \to [k]$ is a balanced colouring, then there is a perfect matching $M$ of $K_{2kn}$ with deviation satisfying
    \[f_c(M)\leq O(k^2).\]
\end{conjecture}

They made some progress towards this conjecture, proving that, in the same setting, $f_c(M) \leq 3k\sqrt{kn \ln (2k)}$.

Further progress was made towards \Cref{conj:balanced-perfect-matching} by the second author \cite{Hol24}, who found a uniform upper bound for $\min_M f_c(M)$, proving that, for any colouring $c$ of $K_{2kn}$, $f_c(M)\leq 4^{k^2}$ for some perfect matching $M$ of $K_{2kn}$. The generalisation to arbitrary $H$ factors was also considered in \cite{Hol24}.
Indeed, we may generalise the definition of deviation from \eqref{eq:matching-deviation} to the following, where $F$ is an $H$-factor of $K_{rnk}$, the graph $H$ has $m$ edges and $r$ vertices, and $c\from E(K_{rnk}) \to [k]$ is a balanced colouring.
\begin{equation}
    \label{eq:deviation}
    f_c(F)\defined \sum_{i=1}^k\abs[\big]{\abs{c^{-1}(i)\inter F} - nm}.
\end{equation}
The suggested generalisation to $H$-factors was as follows.

\begin{question}
\label{question:balanced-factor}
    For which graphs $H$ on $r$ vertices and $m$ edges is there a function $h_H\from \ZZ\to\ZZ$ with the following property?
    
    For any $n$ and any number $k$ of colours, any colour-balanced $c\from K_{rnk}\to [k]$ (that is, such that for every $i$ in $[k]$ there are equally many edges $e$ with $c(e)=i$) contains an $H$-factor $F$ with deviation satisfying
    \begin{equation}
    \label{eq:hH-bound}
        f_c(F)\leq h_H(k).
    \end{equation}
\end{question}

While \Cref{thm:main-k-colouring} resolves \Cref{question:balanced-factor} in its entirety, the question still remains as to the optimal value of the function $h_H(k)$ in \eqref{eq:hH-bound}.
We note here that the bound provided by \Cref{thm:main-k-colouring} is worse than the bound of $4^{k^2}$ from \cite{Hol24} in the case of perfect matchings.
Improving these bounds further will be discussed in \Cref{sec:conclusion}.


\subsection{Outline of the proof of Theorem \ref{thm:main-k-colouring}}
\label{subsec:outline}

We now provide an outline of how we prove \Cref{thm:main-k-colouring}, which consists of deducing this result from \Cref{thm:main}, and then proving the latter.
Throughout this outline, and in the proof itself, we assume that $r$, $k$, and $d$ are constant, and $n$ is sufficiently large.

First, in \Cref{sec:deduction}, to deduce \Cref{thm:main-k-colouring} from \Cref{thm:main}, we simply note that we can embed the palette $[k]$ into $\SS^{k-2}$ as the vertices of a regular simplex.
Deducing that the colouring arising from \Cref{thm:main} is also balanced in the context of \Cref{thm:main-k-colouring} is then a matter of algebraic manipulations.

We prove in \Cref{sec:clique-factor} that \Cref{thm:main} holds when $H$ is the complete graph $K_r$.
In this case, we can use the additional symmetry of $H$ to our advantage.
We define the vector $w$ as the sum of the colours of the edges of the clique-factor $F$ (recalling that the colours are unit vectors), that is
\begin{equation*}
    w \defined \sum_{e\in E(F)} c(e).
\end{equation*}
The proof runs by attempting to minimise the square of the Euclidean norm of $w$ via local changes, swapping where two vertices are embedded; we will refer to the quantity $\norm{w}$ as the \emph{norm of $F$}.
We assign a weight $w_F(e)$ to each edge $e$ by taking the inner product of its colour (a unit vector, we recall) and $w$:
\begin{equation*}
    w_F(e) \defined \inner{c(e)}{w}.
\end{equation*}
The crucial observation about this parameter is that 
\begin{equation*}
    w_F(F) \defined \sum_{e\in F} w_F(e) = \norm{w}^2.
\end{equation*}
This allows us to prove in \Cref{subsec:contradictory-swaps} that, if we start with $H$-factor $F$ and apply a swap to reach $F'$, there is a constant $L$ such that, if $w_F(F) - w_F(F') > L$, then $w_F(F) - w_{F'}(F') > 0$, i.e.\ the norm of $F'$ is less than the norm of $F$.
We may assume that our $K_r$-factor $F$ has minimal norm, and thus call a swap as above ``contradictory'' as it would contradict this assumption.

As the weights in a swap are determined only be the colours of the edges involved, there are only a constant number of potential swaps (looking only at the colours of the edges).
If the norm of $F$ is large, then some of these potential swaps must be contradictory, and so it suffices to show that such a swap may actually be performed.

We may associate swaps with points of $\RR^d$ by taking the signed sum of the colours (which we recall are unit vectors) added and removed by the swap.
In \Cref{subsec:hyperplanes}, we show that we can find a vector $w'$ very close to $w$ such that, if
\begin{equation}
\label{eq:intro-target}
    \sum_{e\in F} \inner{c(e)}{w'} - \sum_{e \in F'} \inner{c(e)}{w'} > 0,
\end{equation}
then $w_F(F) - w_F(F') > L$ and the swap is contradictory.

Then, in \Cref{subsec:finding-contradiction}, we prove that there must be a swap satisfying \eqref{eq:intro-target}.
Due to the symmetry of $K_r$, every edge of the host graph $K_{nr}$ which is present in $F$ is removed from the $K_r$-factor by exactly the same number of swaps.
Likewise, every edge of the host graph which is not in $F$ is added to the $K_r$-factor by exactly the same number of swaps. 
Thus it suffices to show that the average of $\inner{c(e)}{w'}$ over edges $e\in E(F)$ is larger than the average of $\inner{c(f)}{w'}$ over edges $f\in E(K_{nr}) \setminus E(F)$.
This follows from some algebraic manipulations, and so we deduce that there must be a contradictory swap.
The existence of this contradictory swap demonstrates the existence of a $K_r$ factor of $K_{nr}$ whose norm is at most a constant, as required.

In \Cref{sec:simplex-bound}, we consider the case when $S$ is the set of vertices of a regular simplex. We then obtain an explicit constant $C$ for \Cref{thm:main} in the case $H=K_r$, such that if $\norm{w} \geq C$ then there must exist a contradictory swap.

We generalise this to $H$-factors in \Cref{sec:h-factor}.
If $H$ has $r$ vertices, then we first find a $K_r$-factor $F$ using the results of \Cref{sec:clique-factor}, and then embed an $H$-factor inside of $F$.
As there are only a constant number of possible colourings of $K_r$, we may group all copies of $K_r$ in $F$ receiving a given colouring together, and then further subdivide this into blocks of $r!$ copies of $K_r$, with a constant ($< r!$) number of copies left over. 

In each block of $r!$ copies of $K_r$, all with the same colouring, we embed one copy of $H$ into each $K_r$, one with each of the $r!$ possible injections from $V(H)$ to $V(K_r)$.
The left over copies of $H$ are embedded arbitrarily, but there are only a constant number of these.
Thus the $H$-factor inherits the colour-balance of the $K_r$-factor up to a constant additive error, proving \Cref{thm:main}.


\section{Deducing Theorem \ref{thm:main-k-colouring} from Theorem \ref{thm:main}}
\label{sec:deduction}

To prove \Cref{thm:main-k-colouring}, we embed the $k$ colours into a subset of $\SS^{d-1}$ and apply \Cref{thm:main}.
More precisely, let $Q$ be the set of vertices of a regular simplex in $\SS^{k-2}$, so $\card{Q} = k$, and denote the vertices of $Q$ as $q_1, \dots, q_k$.
Let $c\from E(K_{nrk}) \to [k]$ be the $k$-colouring we are given, and $c' \from E(K_{nrk}) \to Q$ be the colouring into the vertices of the simplex.
Then, as the $k$-colouring we are given is completely balanced, we know that the value of $\alpha$ in \Cref{thm:main} is 0, and so we know that there is a constant $C$ depending only on $r$ and $k$, and an $H$-factor $F$ of $K_{nrk}$ such that $\norm{\sum_{e\in E(F)} c'(e)} \leq C$.

Say that there are $b_i$ edges of $E(F)$ coloured $i$ by $c$.
Then the condition on the colouring becomes 
\begin{equation}
\label{eq:condition}
    \norm[\Big]{\sum_{i=1}^k b_i q_i} \leq C, 
\end{equation}
and we need to provide an upper bound on $\sum_{i=1}^k \abs[\big]{b_i - \card{E(F)} / k}$.
However, as $\sum_{i=1}^k q_i = 0$, we can set $b_i' = b_i - \card{E(F)} / k$ so that $\sum_{i=1}^k b_i' = 0$.
Then condition \eqref{eq:condition} remains unchanged, and holds for $b_i'$ as well as for $b_i$, whereas our goal is now to upper-bound $\sum_{i=1}^k \abs{b_i'}$.
As $\sum_{i=1}^k b_i' = 0$, we know that $2\sum_{1\leq i < j \leq k} b_i' b_j' = -\sum_{i=1}^k b_i'^2$.
Squaring \eqref{eq:condition} and noting that, for $i\neq j$, $\inner{q_i}{q_j} = -1/(k-1)$, we find that
\begin{equation*}
    \sum_{i=1}^k b_i'^2 + 2\sum_{1\leq i < j \leq k} b_i' b_j' \inner{q_i}{q_j} = \p[\Big]{1 + \frac{1}{k-1}} \sum_{i=1}^k b_i'^2 \leq C^2.
\end{equation*}
Finally, we can note that $\sum_{i=1}^k \abs{b_i'} \leq \sqrt{k\sum_{i=1}^k b_i'^2}$, and thus deduce
\begin{equation*}
    \sum_{i=1}^k \abs{b_i'} \leq \sqrt{k\sum_{i=1}^k b_i'^2} \leq \sqrt{k-1} \, C,
\end{equation*}
and \Cref{thm:main-k-colouring} follows.
\qed


\section{Finding a balanced clique-factor}
\label{sec:clique-factor}

As described in the above outline, we first prove an intermediate result, which states that we can find a colour-balanced $K_r$-factor.
More precisely, in this section we prove the following.

\begin{theorem}
\label{thm:clique-factor}
    Fix integers $k,r,d\geq 2$ and a subset $S\sseq \SS^{d-1} \sseq \RR^d$ of size $k$ such that $0$ is in the convex hull of $S$.
    For every integer $n$, positive real number $\alpha > 0$, and colouring $c\from E(K_{nr}) \to S$ such that $\norm{\sum_{e\in E(K_{nr})} c(e)} \leq \alpha$, there is a $K_r$-factor $F$ of $K_{nr}$ such that
    \begin{equation*}
        \norm[\bigg]{\sum_{e\in E(F)} c(e) \; } \leq C + \frac{2\alpha}{n},
    \end{equation*}
    where $C$ is a constant which only depends on $r$ and $S$.
    In the case that $S$ is a regular $(d+1)$-simplex (so $k=d+1$), we may take
    \begin{equation*}
        C = (8dr)^{(8dr)^{d+1} - 1}.
    \end{equation*}
\end{theorem}

Firstly, note that if $S$ does not span $\RR^d$, then we may pass to a linear subspace of $\RR^d$ isomorphic to $\RR^{d'}$, for some $d' < d$, in which $S$ is spanning.
As all of our bounds are increasing in $d$, this allows us to assume without loss of generality that $S$ spans $\RR^d$.

To minimise $\norm[\big]{\sum_{e\in E(F)} c(e)}$ for $F$ a $K_r$-factor, we break up this norm into a sum of inner products.
Indeed, the following definitions are crucial to our arguments.

\begin{definition}
    Define the \emph{weight} of a colour $s\in S$ with respect to a $K_r$-factor $F\sseq K_{nr}$ as
    \begin{equation*}
        w_F(s) \defined \inner{s}{w} = \sum_{f\in E(F)} \inner{s}{c(f)},
    \end{equation*}
    where
    \begin{equation*}
        w \defined \sum_{f\in E(F)} c(f).
    \end{equation*}
    The weight of an edge $e\in E(K_{nr})$ is then the weight of the colour of that edge: $w_F(e) = w_F(c(e))$.
    We interpret the weight of a set $X\sseq E(K_{nr})$ as $w_F(X) = \sum_{e\in X} w_F(e)$.
\end{definition}

We may note that the weight of the whole $K_r$-factor $F$ is
\begin{equation*}
    w_F(F) = \sum_{e\in E(F)} w_F(e) = \! \sum_{e,f\in E(F)} \! \inner{c(e)}{c(f)} = \norm[\bigg]{\sum_{e\in E(F)} c(e)}^2 = \norm{w}^2.
\end{equation*}

The above equality, despite its simplicity, is in some ways the crux of our entire argument, as will be expanded on in \Cref{subsec:contradictory-swaps}.
Assume that we have a $K_r$-factor $F$ of minimal weight; there can thus be no perturbation of $F$ which reduces the total weight.


\subsection{Contradictory swaps}
\label{subsec:contradictory-swaps}

We will now consider small perturbations of our $K_r$-factor $F$ which may decrease the weight.
In particular, we will consider swaps of two vertices.

\begin{definition}
    \label{def:swap}
    For $u,v$ vertices in distinct components of $F$, we define the \emph{swap} $u \swap v$ to be the following operation: For each vertex $w \neq u$ in the component of $u$, remove the edge $uw$ and add the edge $vw$, and for each vertex $w' \neq v$ in the component of $v$, remove the edge $vw'$ and add the edge $uw'$.
    Write the resulting $K_r$-factor as $F_{u\swap v}$

    Call the swap $u\swap v$ \emph{contradictory} if $w_{F_{u\swap v}}(F_{u\swap v}) < w_F(F)$ (noting that this will contradict our assumption that $F$ has minimal weight).
\end{definition}

Note that any swap removes precisely $2r-2$ edges from $F$, and adds $2r-2$ distinct edges to form $F_{u\swap v}$.
While it may not be obvious at first how one may go about showing that a swap is contradictory in practice, as the weight function $w_F$ has changed along with the $K_r$-factor itself, we now show that, in fact, one can tell whether a swap is contradictory only in terms of the weight function $w_F$ (or equivalently the vector $w$ for $F$).

\begin{lemma}
\label{lem:contradictory-criterion}
    Let $u,v$ be vertices of $K_{nr}$ in different components of $F$, and let $E_1$ and $E_2$ be the sets of edges incident to either $u$ or $v$ in $F$ and $F'=F_{u\swap v}$ respectively.
    Then, if
    \begin{equation}
    \label{eq:contradictory-criterion}
        \sum_{e\in E_1} w_F(e) - \sum_{e\in E_2} w_F(e) > L,
    \end{equation}
    where $L = 8(r-1)^2$, then the swap $u\swap v$ is contradictory.
\end{lemma}

We note before giving the proof of \Cref{lem:contradictory-criterion} that the condition \eqref{eq:contradictory-criterion} does not refer to $w_{F'}$, only $w_F$.

\begin{proof}
    We compute the difference $w_F(F) - w_{F'}(F')$.
    For ease of notation, for any sets $X,Y\sseq E(K_{nr})$ write $w_F(X) = \sum_{e\in X}w_F(e)$ and
    \[\inner{X}{Y} \defined \sum_{e\in X} \sum_{f\in Y} \inner{c(e)}{c(f)}.\]
    Let $E_0 \defined E(F) \setminus E_1 = E(F') \setminus E_2$.
    We then have
    \begin{align*}
        w_F(F) - w_{F'}(F') 
        &= \inner{E(F)}{E(F)} - \inner{E(F')}{E(F')} \\
        &= \inner{E_1}{E_1} + 2\inner{E_0}{E_1} - \inner{E_2}{E_2} - 2\inner{E_0}{E_2} \\
        &= 2(\inner{E_1}{E_0\union E_1} - \inner{E_2}{E_0 \union E_1}) - \inner{E_1}{E_1} - \inner{E_2}{E_2} + 2\inner{E_1}{E_2} \\
        &= 2(w_F(E_1) - w_F(E_2)) - \norm[\Big]{\sum_{e\in E_1\union E_2} \!\! c(e) \; }^2.
    \end{align*}
    Note that $\norm[\Big]{\sum_{e\in E_1\union E_2} c(e)}^2 \leq (4(r-1))^2$ as $\card{E_1} = \card{E_2} = 2(r-1)$.
    Thus, if $w_F(E_1) - w_F(E_2) > 8(r-1)^2$, then the swap $u \swap v$ is contradictory, and the lemma is proved.
\end{proof}


\subsection{Partitioning the space of swaps}
\label{subsec:hyperplanes}

The only information we need to determine whether or not a swap is contradictory is the (multi-)set of colours of edges removed and of edges added by the swap.
As a swap removes $2(r-1)$ edges and adds $2(r-1)$ edges, there are therefore only a constant number of possible swaps considering only the colours of edges removed and added.
Indeed, fix an arbitrary enumeration $S = \set{s_1,\dotsc,s_k}$, and set
\begin{equation*}
    Y\defined \set{y \in \set{0,1,\dotsc,2(r-1)}^k \st y_1+\dotsb+y_k = 2(r-1)},
\end{equation*}
thinking of $(y_1,\dotsc,y_k)$ as representing a collection of edges of which $y_i$ have colour $s_i$.

A swap moves between two vertices of $Y$.
The only relevant information for determining whether a swap satisfies \eqref{eq:contradictory-criterion} is the relative counts of colours between the start and end points of the swap, and so we may encode swaps as points in the space
\begin{equation*}
    X\defined \set{x \in \set{2-2r,3-2r,\dotsc,-1,0,1,\dotsc,2r-3,2r-2}^k \st \sum_{i=1}^k x_i = 0, \sum_{i=1}^k \abs{x_i} \leq 4r-4 }.
\end{equation*} 

The \emph{weight} of a swap $x\in X$---the difference $w_F(E_1) - w_F(E_2)$ as in \eqref{eq:contradictory-criterion}---is then given by
\begin{equation*}
    \sum_{i=1}^k x_i w_F(s_i) = \inner{g(x)}{w},
\end{equation*}
where
\begin{equation}
\label{eq:define-g}
    g(x) \defined \sum_{i=1}^k x_i s_i.
\end{equation}
We now introduce some more terminology concerning these swaps.
\begin{definition}
    For $L$ as in \Cref{lem:contradictory-criterion}, a swap $x\in X$ is 
    \begin{itemize}
        \item \emph{small} if $\absinner{w}{g(x)} \leq L$,
        \item \emph{increasing} if $\inner{w}{g(x)} > L$, and
        \item \emph{decreasing} if $\inner{w}{g(x)} < -L$.
    \end{itemize}
    Let $X_S\sseq X$ be the set of all small swaps.    
\end{definition}

Note that \Cref{lem:contradictory-criterion} states precisely that all decreasing swaps are contradictory.
Intuitively, if the swap between $y,y'\in Y$ is small, then $g(y')$ lies close to the hyperplane through $g(y)$ with normal vector $w$ (and vice versa).
We will now perturb this hyperplane slightly to produce a new hyperplane, $W$, so that all small swaps go between two points in the same translate of $W$.
Before we work towards a definition of $W$, we first need another definition.

\begin{definition}
\label{def:angle}
    We define the \emph{angle} $\theta(x)$ of a swap $x\in X$ as follows:
    \begin{equation*}
        \theta(x) \defined \sin^{-1}\p[\bigg]{\frac{\abs{\inner{g(x)}{w}}}{\norm{g(x)}\norm{w}}},
    \end{equation*}
    if $g(x) \neq 0$; otherwise we set $\theta(x) = 0$.
    i.e.\ $\theta(x)$ is the absolute value of the angle between $g(x)$ and the hyperplane with normal $w$.
\end{definition}

The hyperplane $W$ will contain a set of the form $W_\theta \sseq X$ for some angle $\theta$, where
\begin{equation}
\label{eq:def_w_theta}
    W_\theta\defined\set[\bigg]{x\in X \st \frac{\absinner{g(x)}{w}}{\norm{g(x)}\norm{w}} < \sin(\theta)}.
\end{equation}
Note that $W_\theta$ is the set of swaps with angle to $w^\perp$ less than $\theta$.
$W_\theta^\perp$ is then defined to be the space orthogonal to the span of $W_\theta$.

The rest of this subsection is dedicated to finding an angle $\theta$ for which $W_\theta$ may be extended to a hyperplane with the following desired properties.
The core ideas of the argument are that, if $\theta$ is sufficiently small, then the codimension of $W_\theta$ is at least 1.
Moreover, as $\theta\to 0$, the angle between $w$ and $W_\theta^\perp$ also tends to 0.
Thus, for sufficiently small $\theta$, the set $W_\theta$ can be extended to a hyperplane which is a good approximation of $w^\perp$.

If $\theta$ is also not too small in terms of $\norm{w}$, then all small swaps will lie in $W_\theta$.
This will give us some range $[\theta_0,\theta_1]$ of values for $\theta$ in which all the above properties hold.
The final condition we require is that there are no points of $x\in X$ so that $g(x)$ lies ``just outside'' of $W_\theta$, in the sense that $g(x) \in W_{\theta(1 + \zeta)}$ for some not-too-large $\zeta$.
If $\theta_1/\theta_0$ is sufficiently large, then we will show that this condition can be made to hold by a simple pigeonhole argument.
We now formalise the above intuition in a series of lemmas.

Throughout all the following lemmas, $S$ is any finite subset of $\SS^{d-1}$ of codimension zero, and $r$ is an integer.
We first provide the statements, and then move on to the proofs.

\begin{lemma}
\label{lem:angle-small-swaps}
    There is a constant $\beta_0 = \beta_0(S,r) > 0$ such that, if $\sin(\theta) \geq \beta_0 / \norm{w}$, then all small swaps are in $W_\theta$.
\end{lemma}

\begin{lemma}
\label{lem:angle-codim}
    There is a constant $\beta_1 = \beta_1(S,r) > 0$ such that, for all $\theta < \sin^{-1}(\beta_1)$, the codimension of $W_\theta$ in $\RR^d$ is at least 1.
\end{lemma}

\begin{lemma}
\label{lem:angle-relation}
    There is a constant $\eta = \eta(S,r) > 0$ such that, if $\phi = \phi(\theta)$ is the angle between $w$ and the space $W_\theta^\perp$, then $\sin(\phi) \leq \eta \sin(\theta)$.
\end{lemma}

\begin{observation}
\label{obs:angle-jump}
    For all $\eta\geq 1$, if $\norm{w}\beta_1 / \beta_0 \geq \eta^{\card{X}}$ then there is some angle $\theta$ with 
    \[\frac{\beta_0}{\norm{w}} \leq \sin(\theta) \leq \frac{\beta_0\eta^{\card{X}-1}}{\norm{w}}\]
    such that, for all $x\in X\setminus W_\theta$, the angle $\theta(x)$ between $g(x)$ and $w^\perp$ satisfies
    \[sin(\theta(x)) \geq \eta\sin(\theta)\]
    
\end{observation}

Indeed, \Cref{obs:angle-jump} follows from considering $X = \set{x_1,x_2,\dotsc,x_{\card{X}}}$ in increasing order of $\theta(x)$: if $j$ is maximal such that $x_j \in W_{\theta}$, then either $\sin(\theta(x_{\card{X}})) / \sin(\theta(x_j)) \leq \eta^{\card{X} - 1}$, or there is some $i \geq j$ such that $\sin(\theta(x_{i+1})) / \sin(\theta(x_i)) > \eta$.

We now prove \Cref{lem:angle-codim,lem:angle-relation,lem:angle-small-swaps}.

\begin{proof}[Proof of \Cref{lem:angle-small-swaps}]
    Let $x\in X$ be a small swap. 
    We define the constant
    \begin{equation}
        \label{eq:beta_0-definition}
        \beta_0 = \frac{L}{\min\set{\norm{g(x)} \st x \in X, \, g(x) \neq 0}}.
    \end{equation}
    Then, by definition, $\absinner{g(x)}{w} \leq L$ and $\absinner{g(x)}{w} = \norm{g(x)} \norm{w} \sin(\theta(x))$.
    Therefore 
    \[\sin(\theta(x)) \leq \frac{L}{\norm{w}\norm{g(x)}} \leq \frac{\beta_0}{\norm{w}} \leq \sin(\theta).\]
    We therefore see that $\theta(x) < \theta$, and so $x$ is in $W_\theta$, as required.
\end{proof}

\begin{proof}[Proof of \Cref{lem:angle-codim}]
    There is a finite list of subsets of $g(X)$ with codimension 0.
    For each such $J\sseq g(X)$, we claim that there is some $\eps_J > 0$ such that, for any unit vector $v$, there is some $x\in J$ with $\absinner{v}{x} \geq \eps_J$.
    
    Indeed, let $f\from \SS^{d-1} \to \RR$ be given by $f(v) = \max_{x\in J} \absinner{v}{x}$.
    Then $f$ is a continuous function on a compact space and so attains its minimum: let $v_0\in \SS^{d-1}$ be such that $f(v_0)$ is minimal.
    It thus suffices to prove that $f(v_0) > 0$.
    It is immediate that $f(v_0) \geq 0$, and $f(v_0)=0$  would imply that $\inner{x}{v_0} = 0$ for all $x\in J$, contradicting the fact that $J$ has codimension 0.
    Thus $\eps_J > 0$.

    Then we take 
    \[\beta_1 = \frac{\min_J (\eps_J)}{\max_{x\in X\setm{0}} \norm{g(x)}}.\]
    Now assume for contradiction that, for some $\theta$ with $0 < \theta < \sin^{-1}(\beta_1)$, the set $W_\theta$ has codimension 0.
    Thus there is some $J\sseq g(X)$ of codimension 0 with $J \sseq W_\theta$.
    Therefore, by definition of $W_\theta$, all $x\in J$ have $\absinner{x}{w} < \norm{x}\norm{w}\abs{\sin(\theta)}$.
    But we know from the definition of $\eps_J$ that there is some $x\in J$ for which $\absinner{x}{\widehat{w}} \geq \eps_J$.
    Therefore
    \[\eps_J \leq \absinner{x}{\widehat{w}} < \norm{x}\abs{\sin(\theta)} \leq \norm{x}\beta_1 \leq \min_J(\eps_J),\]
    a contradiction, as required.
\end{proof}

\begin{proof}[Proof of \Cref{lem:angle-relation}]
    There is a finite list of subsets of $g(X)$ with codimension at least $1$.
    It thus suffices to prove that, for an arbitrary $J\sseq g(X)$ spanning some $U\leq \RR^d$ of codimension at least 1, there is a constant $\eta_J > 0$ such that, if $J\sseq W_\theta$ and the span of $W_\theta$ is also $U$, then $\sin(\phi) \leq \eta_J \sin(\theta)$.
    
    We claim that there is some some $\eta_J > 0$ such that, for any unit vector $u\in U$, there is some $x\in J$ for which 
    \begin{equation}
    \label{eq:angle-relation-claim}
        \absinner{u}{x} > \eta_J^{-1} \norm{x}.
    \end{equation}
    Indeed, this follows by similar reasoning to that used in the proof of \Cref{lem:angle-codim}:
    let $B$ be the boundary of the Euclidean unit ball in $U$ (i.e.\ the set of unit vectors), and define $f\from B\to \RR$ by $f(v) = \max_{x\in J} \absinner{v}{x}$.
    Then $f$ is a continuous function on a compact space, so attains its minimum, and the claim follows.

    It remains to show that $\sin(\phi) \leq \eta_J \sin(\theta)$ under the assumption that $J \sseq W_\theta$ and these two sets have the same span.
    
    As $U$ has codimension at least 1, there is a space $U^\perp \leq \RR^d$ of dimension at least 1.
    Let $\pi_U$ be the orthogonal projection onto $U$ and note that, by definition, $\phi$ is the angle between the vector $w$ and $U^\perp$.
    It follows that $\norm{\pi_U(\widehat{w})} = \sin(\phi)$.
    As $J \sseq W_\theta$ by assumption, for all $x\in J$ we have 
    \[\absinner{x}{\pi_U(\widehat{w})} = \absinner{x}{\widehat{w}} \leq \norm{x}\sin(\theta).\]
    
    Thus, defining $v = \pi_U(\widehat{w}) / \sin(\phi)$, for all $x\in J$ we have 
    \begin{equation}
    \label{eq:angle-relation-ub}
        \absinner{x}{v} \leq \norm{x}\sin(\theta) / \sin(\phi).
    \end{equation}

    Combining \eqref{eq:angle-relation-claim} and \eqref{eq:angle-relation-ub}, we deduce that $\eta_J^{-1} < \sin(\theta) / \sin(\phi)$. Letting $\eta = \max_J (\eta_J)$, the desired result follows.
\end{proof}

We may now define the hyperplane $W\sseq \RR^d$.
Assume that $\norm{w}\beta_1 / \beta_0 \geq \eta^{\card{X}}$, and $\theta$ is chosen as in \Cref{obs:angle-jump}.
Let $U$ be the space orthogonal to both $W_\theta$ and $w$.
Then $W$ is the space spanned by $W_\theta \union U$.
We know from \Cref{lem:angle-codim} that $W_\theta$ has codimension at least 1, and from \Cref{lem:angle-relation} that $w\notin W_\theta$.
Thus $W$ is indeed a hyperplane in $\RR^d$.
Define $w'$ to be the vector normal to $W$ with $\norm{w'} = \norm{w}$.

Let $\cZ$ be the partition of $Y$ by translates of $W$.
We now prove the following lemma.

\begin{lemma}
\label{lem:total-order}
    Let $Z_1,Z_2\in \cZ$ be distinct, and let $y_1 \in Z_1$ and $y_2\in Z_2$.
    Then $y_1 - y_2$ is an increasing swap if and only if $\inner{y_1 - y_2}{w'} > 0$.
\end{lemma}

\begin{proof}
    Assume for contradiction that $y_1 - y_2$ is increasing, i.e.\ $\inner{y_1 - y_2}{w} > L$, but $\inner{y_1 - y_2}{w'} < 0$.
    Let $P$ be the plane spanned by the vectors $w'$ and $y_1 - y_2$, and let $\pi \from \RR^d \to P$ be the orthogonal projection to $P$.
    Then, letting $\phi'$ be the angle between $\pi(w)$ and $w'$, as $\phi < \pi / 2$, we have that $\phi' \leq \phi$.

    Note that the angle between the vector $y_2 - y_1$ and $w$ is more than $\pi / 2$ and the angle between $y_2 - y_1$ and $w'$ is less than $\pi / 2$.
    Therefore the angle between $y_2 - y_1$ and $w$ is less than $\pi / 2 + \phi'$, and thus the angle between $y_2 - y_1$ and $w^\perp$ is less than $\phi' < \phi$.
    
    However, we know that $\sin(\phi) \leq \eta \sin(\theta)$, and so, by \Cref{obs:angle-jump}, we in fact have that the angle between $y_1 - y_2$ and $w^\perp$ is at most $\theta$.
    Thus $y_1 - y_2 \in W_\theta$, contradicting the fact that $Z_1$ and $Z_2$ are distinct.
\end{proof}

Note that \Cref{lem:angle-codim} implies that $\card{\cZ} > 1$, and \Cref{lem:total-order} implies that $\cZ$ inherits a total order $\prec$ from the swaps (noting that the value of $\inner{y_1 - y_2}{w'}$ is independent of the choices of $y_i \in Z_i$); write $Z_1\prec Z_2$ if every swap from $Z_1$ to $Z_2$ is increasing.
For any such $Z_1$ and $Z_2$, a swap from $Z_2$ to $Z_1$ is contradictory.
In the next section, we will prove that there must be such a swap.


\subsection{Finding a contradictory swap}
\label{subsec:finding-contradiction}

We know from the definition of $\cZ$ that each $Z\in \cZ$ is the set of vectors $x$ such that $\inner{w'}{x} = p$ for some $p$.
Assume for contradiction that there is no contradictory swap between distinct elements of $\cZ$.

Thinking now back in terms of our $K_r$-factor $F$, we see that, due to \Cref{lem:total-order}, every swap which replaces edges $E_1$ with $E_2$ must have $\inner{w'}{E_1} \leq \inner{w'}{E_2}$, where $\inner{v}{X}$ is shorthand for $\sum_{e\in X} \inner{v}{c(e)}$.
By symmetry of $K_r$, the number of swaps $u\swap v$ which remove an edge $e\in E(F)$ does not depend on $e$, and likewise for adding an edge of $E(K_{nr})\setminus E(F)$.
Thus, comparing average values of $\inner{w'}{c(e)}$, we find that
\[\frac{1}{e(F)} \sum_{e\in E(F)} \inner{w'}{c(e)} \leq \frac{1}{e(K_{nr}) - e(F)} \sum_{e\in E(K_{nr})\setminus E(F)} \inner{w'}{c(e)}.\]
Noting that $e(F) \leq nr(r-1)/2$, and recalling that $w = \sum_{e\in E(F)} c(e)$, we deduce
\[\inner{w'}{w} \leq \frac{r-1}{r(n-1)}\p[\big]{\inner{w'}{E(K_{nr})} - \inner{w'}{w}}.\]
Recalling our assumption that $\norm[\big]{\sum_{e\in E(K_{nr})} c(e)} \leq \alpha$, we see that to deduce a contradiction it suffices to prove that
\[\p[\bigg]{1 + \frac{r-1}{r(n-1)}}\inner{w'}{w} > \frac{1}{n}\norm{w'}\alpha.\]
In particular, we know from \Cref{obs:angle-jump} that the angle $\phi$ between $w$ and $w'$ satisfies $\sin(\phi) \leq \beta_0\eta^{\card{X}} / \norm{w}$, and so it suffices to prove that
\[\norm{w} \cos(\phi) > \frac{\alpha}{n}.\]

If $\norm{w} > 2\beta_0 \eta^{\card{X}}$ (which is automatically true by our assumption that the conditions of \Cref{obs:angle-jump} hold, provided we take $\beta_1 \leq 1/2$, which we may do) then $\cos(\phi) > 1/2$, so it suffices that $\norm{w} > 2\alpha / n$.

As $\theta_0$, $\eta$, and $\card{X}$ are constants, we may therefore see that there is some constant $C$ such that \Cref{thm:clique-factor} holds.


\section{Finding explicit bounds when \texorpdfstring{$S$}{S} is a simplex}
\label{sec:simplex-bound}

In this section, we consider only the case when $k=d+1$ and $S\sseq \SS^{d-1}$ consists of the vertices of a regular simplex. 
This corresponds precisely to the case in \Cref{thm:main-k-colouring}, as shown in \Cref{sec:deduction}. 
Our goal is to find an explicit value for the constant $C$ in \Cref{thm:clique-factor} in this case.

To do this, we find explicit values for the constants $\beta_0$, $\beta_1$, and $\eta$ which we do in \Cref{lem:angle-codim}, \Cref{lem:angle-relation}, and \Cref{lem:angle-small-swaps} respectively.

\begin{lemma}
    \label{lem:angle-small-swaps-bound}
    If $S$ is a regular simplex, in \Cref{lem:angle-small-swaps} we may take $\beta_0 = L = 8(r-1)^2$.
    That is, if $\theta$ is such that $\sin(\theta) \geq L/\norm{w}$, then all small swaps are in $W_\theta$.
\end{lemma}

\begin{lemma}
\label{lem:angle-codim-bound}
    If $S$ is a regular simplex, in \Cref{lem:angle-codim} we may take $\beta_1 = (8dr)^{-d}$.
    That is, for $\beta_1$ as above, for all $\theta < \sin^{-1}(\beta_1)$, the codimension of $W_\theta$ in $\RR^d$ is at least 1.
\end{lemma}

\begin{lemma}
\label{lem:angle-relation-bound}
    If $S$ is a regular simplex, in \Cref{lem:angle-relation} we may take $\eta = (8dr)^{3d^2}$.
    That is, for $\eta$ as above and $\phi$ the angle between $w$ and $W_\theta^\perp$, we have $\sin(\phi) \leq \eta\sin(\theta)$.
\end{lemma}

\begin{proof}[Proof of \Cref{lem:angle-small-swaps-bound}]
    Recalling the definition of $\beta_0$ in \eqref{eq:beta_0-definition}, we see that it suffices to prove that $\min\set{\norm{g(x)} \st x \in X, \, g(x) \neq 0} \geq 1$.
    In fact, we claim that 
    \begin{equation}
    \label{eq:min-lattice-distance}
        \min\set{\norm{g(x)} \st x \in X, \, g(x) \neq 0} = \sqrt{2+\tfrac{2}{d}}.    
    \end{equation}
    Indeed, for arbitrary non-zero $x\in X$, we have
    \[\norm{g(x)}^2 = \inner[\bigg]{\sum_{i=1}^{d+1} x_i s_i}{\sum_{i=1}^{d+1} x_i s_i} = \sum_{i=1}^{d+1}\sum_{j=1}^{d+1}x_i x_j \inner{s_i}{s_j}.\]
    We know that
    \begin{equation*}
        \inner{s_i}{s_j} = \begin{cases}
            1 &\text{ if } i=j,\\
            -\tfrac{1}{d} &\text{ otherwise.}
        \end{cases}
    \end{equation*}
    We may also use the fact that $\sum_{i=1}^{d+1} x_i = 0$ to deduce that 
    \[2\sum_{1\leq i < j \leq d+1} x_i x_j = -\sum_{i=1}^{d+1} x_i^2.\]
    It therefore follows that
    \[\norm{g(x)}^2 = \p[\big]{1+\tfrac{1}{d}} \sum_{i=1}^{d+1} x_i^2 \geq 2 + \tfrac{2}{d},\]
    as required for \eqref{eq:min-lattice-distance} to hold.
    We may in particular note that $\sqrt{2 + \tfrac{2}{d}} \geq 1$, and so we do indeed find that $\beta_0 = L$ suffices, completing the proof of \Cref{lem:angle-small-swaps-bound}.
\end{proof}

\begin{proof}[Proof of \Cref{lem:angle-codim-bound}]
    Assume for contradiction that $\theta < \sin^{-1}(\beta_1)$ and the codimension of $W_\theta$ is $0$.
    Then there are points $x_1,\dotsc,x_d\in X$ such that, for all $i$, we have $\theta(x_i) \leq \theta$, and the parallelotope $P$ with vertex set
    \[V(P) = \set[\bigg]{\sum_{i=1}^d\xi_i g(x_i) \st \xi \in \set{-1,+1}^d},\]
    has non-zero volume in $\RR^d$.

    We first bound $\vol{P}$ from below. 
    We note that the set
    \begin{equation*}
        \set[\Big]{\sum_{x\in X} \lambda_x g(x) \st \lambda \in \ZZ^X}
    \end{equation*}
    forms a lattice $\Lambda$, and we know from \eqref{eq:min-lattice-distance} that the minimal distance between two points of $\Lambda$ is greater than 1, and all vertices of $P$ lie in $\Lambda$. 
    The volume of $P$ is minimised when it is a fundamental domain of $\Lambda$. 
    We know that the volume of a regular simplex in $\RR^d$ with unit side length is $\sqrt{(d+1)2^{-d}}/d!$, and so when the side length is $\sqrt{2(d+1)d^{-1}}$, the volume is $\sqrt{(d+1)^{d+1}d^{-d}}/d!$.
    The fundamental domain is a parallelotope with volume precisely $d!$ times that of this simplex, and thus 
    \[\vol{P} \geq \sqrt{(d+1)^{d+1}d^{-d}} > 1.\]

    We now bound $\vol{P}$ from above by the volume of its cross-section along the normal to $w$ times the width in direction $w$.
    We now claim that
    \begin{equation}
    \label{eq:max-g-dist}
        \max\set{\norm{g(x)} \st x \in X} \leq 4r.
    \end{equation}
    Indeed, we know by definition that $g(x_i) = \sum_{j=1}^{d+1} x_i(j)s_j$, where we are writing $x_i(j)$ to represent the $j^\text{th}$ coordinate of $x_i$. Then, the facts that $\norm{s_j} = 1$ and $\sum_{j=1}^{d+1}\abs{x_i(j)} \leq 4(r-1) \leq 4r$ allow us to conclude \eqref{eq:max-g-dist}.
    
    Thus the furthest some vertex $\sum_{i=1}^d \xi_i g(x_i)$ of $P$ can be from $0$ is at most $4dr$, and so the cross-section of $P$ has volume at most $(8dr)^{d-1}$.

    For all $x_i$, since $\theta(x_i) \leq \sin^{-1}(\beta_1)$, we have from \eqref{eq:def_w_theta} that
    \[ \frac{\absinner{g(x_i)}{w}}{\norm{g(x_i)}{\norm{w}}} \leq \sin^{-1}(\beta_1).\]
    Combining the above with inequality \eqref{eq:max-g-dist}, we find that $\absinner{g(x_i)}{w} \leq 4r \norm{w} \sin^{-1}(\beta_1)$. However, every vertex of $P$ is a sum of at most $d$ of these $x_i$ so if $x$ is a vertex of $P$, we have $\absinner{x}{w} \leq 4dr \norm{w} \sin^{-1}(\beta_1)$. The width of $P$ in direction $w$ is thus at most $8dr \sin^{-1}(\beta_1)$. We thus have
    \[1 < \vol{P} \leq (8dr)^d \sin^{-1}(\beta_1),\]
    and hence, since $\beta_1 = (8dr)^{-d}$, we obtain a contradiction and thus the codimension of $W_\theta$ must be at least $1$. Thus $\beta_1 = (8dr)^{-d}$ suffices (and we note that $\beta_1 \leq 1/2$, so the assumption of \Cref{obs:angle-jump} is satisfied, as discussed at the end of \Cref{sec:clique-factor}.)
\end{proof}

\begin{proof}[Proof of \Cref{lem:angle-relation-bound}]
    We first note that 
    \begin{equation}
    \label{eq:sin_phi_equals}
        \sin(\phi) = \frac{\norm{\pi_{W_\theta}(w)}}{\norm{w}},
    \end{equation}    
    and so we will find an upper bound on $\norm{\pi_{W_\theta}(w)}$.
    Indeed, if $y_1, \dotsc,y_m$ is an orthonormal basis for $\text{span}(W_\theta)$, then 
    \begin{equation}
    \label{eq:norm_pi_W_bound}
        \norm{\pi_{W_\theta}(w)} \leq \sum_{i=1}^m \absinner{y_i}{w}.
    \end{equation}
    We therefore take some points $x_1,\dotsc,x_m\in X$ such that $g(x_1),\dotsc,g(x_m)$ are a basis for $W_\theta$ and apply Gram--Schmidt orthogonalisation so that $y_i = \sum_{j\leq i} a_{ij} g(x_j)$ is an orthogonal basis, and track upper bounds for the coefficients $a_{ij}$.
    These upper bounds will be derived in turn from a parameter $D$, defined as follows.

    We let $D$ be the minimum distance from any point in $g(X)$ to any hyperplane spanned by points in $g(X)$. 
    This is the minimum height of a non-degenerate simplex in $g(X)$ of any dimension at most $d$. 
    Since the base of such a simplex consists of $d-1$ points of norm at most $4r$, the base has area at most $(8r)^{d-1}$.
    We recall from the proof of \Cref{lem:angle-codim-bound} that the volume of such a simplex is strictly greater than $1/d!$, and so the height is at least $(d!(8r)^{d-1})^{-1}$. 
    Thus 
    \begin{equation}
    \label{eq:D_bound}
        D \geq \frac{1}{d!(8r)^{d-1}} \geq (8dr)^{-d}.
    \end{equation}

    Let $g(x_1),\dotsc,g(x_m)$ be a basis for $W_\theta$ consisting of points in $g(X)$. 
    We now apply Gram-Schmidt orthonormalisation to obtain an orthonormal basis $y_1,\dotsc,y_m$ for $W_\theta$ satisfying 
    \begin{equation}
    \label{eq:def_of_a}
        y_i = \sum_{j\leq i} a_{ij} g(x_j)
    \end{equation} 
    for some constants $a_{ij}$. 
    We seek to bound these coefficients. 

    Let $b_{ij} = \inner{g(x_i)}{y_j}$, so that $g(x_i) - \sum_{j<i} b_{ij} y_j$ is orthogonal to all of $y_1,\dotsc,y_{i-1}$, and so we have, for all $i$, that
    \begin{equation}
    \label{eq:yi_expression}
        g(x_i) - \sum_{j<i} b_{ij} y_j = \norm[\Big]{g(x_i) - \sum_{j<i} b_{ij} y_j} y_i
    \end{equation}
    is a multiple of $y_i$. 
    Since 
    \begin{equation}
    \label{eq:norm_g_x_bound}
        \norm{g(x_i)} \leq \sum_{j=1}^{d+1} \abs{x_i(j)} \, \norm{s_j} \leq 4r
    \end{equation}
    and $y_j$ is a unit vector, we have 
    \begin{equation}
    \label{eq:b_bound}
        \abs{b_{ij}} \leq 4r \; \text{ for all } \; i,j
    \end{equation}    
    We may then note that, for a fixed index $p$, the vector $g(x_p) - \sum_{i<p} b_{pi} y_i$ is the altitude of a simplex in $g(X)$, and so 
    \begin{equation}
    \label{eq:norm_bound}
        \norm[\Big]{g(x_p) - \sum_{i<p} b_{pi} y_i} \geq D.
    \end{equation}
    We may now rewrite equality \eqref{eq:yi_expression} using \eqref{eq:def_of_a} to find that
    \begin{equation}
    \label{eq:expanded_yp}
        \norm[\Big]{g(x_p) - \sum_{i<p} b_{pi} y_i} \sum_{i\leq p}a_{pi} \, g(x_i) = g(x_p) - \sum_{i<p} b_{pi} y_i = g(x_p) - \sum_{i<p}\sum_{j\leq i}b_{pi}\,a_{ij} \,g(x_j).
    \end{equation}
    Recalling that $g(x_1),\dotsc,g(x_m)$ is a basis, and therefore linearly independent, we may equate coefficients of $g(x_i)$ in \eqref{eq:expanded_yp}.
    We may thus deduce that, for all fixed $q < p$, we have
    \begin{equation}
    \label{eq:coefficient_comparison}
        a_{pq} \norm[\Big]{g(x_p) - \sum_{i<p} b_{pi} y_i} = -\sum_{i=q}^{p-1} b_{pi}a_{iq}.
    \end{equation}
    Moreover, by considering the coefficients of $g(x_p)$ in \eqref{eq:expanded_yp}, we may deduce that
    \begin{equation*}
        \norm[\Big]{g(x_p) - \sum_{i<p} b_{pi} y_i} a_{pp} = 1,
    \end{equation*}
    and so inequality \eqref{eq:norm_bound} tells us that
    \begin{equation}
    \label{eq:diagonal_bound}
        \abs{a_{pp}} \leq D^{-1}.
    \end{equation}
    We may now apply inequalities \eqref{eq:b_bound} and \eqref{eq:norm_bound} to deduce from \eqref{eq:coefficient_comparison} that, again for all $p$ and $q$ with $q < p$,
    \begin{equation}
    \label{eq:recursive_bound}
        \abs{a_{pq}} \leq 4rD^{-1} \sum_{i=q}^{p-1}\abs{a_{iq}} \leq 4prD^{-1} \max_{i < p} \abs{a_{iq}}.
    \end{equation}
    A simple induction applied to \eqref{eq:recursive_bound} then shows us that
    \begin{equation}
    \label{eq:a_pq_bound}
        \abs{a_{pq}} \leq (4prD^{-1})^{p-q}\abs{a_{qq}} \leq \p[\big]{4mrD^{-1}}^m,
    \end{equation}
    where we have used from \eqref{eq:diagonal_bound} that $\abs{a_{qq}} \leq 4mrD^{-1}$.

    We can now combine the above results to produce the desired upper-bound on $\sin(\phi) / \sin(\theta)$.

    First, we may rearrange the definition \eqref{eq:def_w_theta} of $W_\theta$ and apply inequality \eqref{eq:norm_g_x_bound} to find that, for all $i$,
    \begin{equation*}
        \sin(\theta) \geq \frac{\absinner{g(x_i)}{w}}{\norm{g(x_i)}\norm{w}} \geq \frac{\absinner{g(x_i)}{w}}{4r\norm{w}}.  
    \end{equation*}
    Next, we may use inequality \eqref{eq:norm_pi_W_bound} along with equality \eqref{eq:sin_phi_equals} to find that
    \begin{equation*}
        \sin(\phi) = \frac{\norm{\pi_{W_\theta}(w)}}{\norm{w}} \leq \frac{\sum_{j=1}^m \absinner{y_j}{w}}{\norm{w}},
    \end{equation*}
    and thus, combining the above two inequalities,
    \begin{equation}
    \label{eq:sin_phi_over_sin_theta_bound}
        \frac{\sin(\phi)}{\sin(\theta)} \leq \frac{4r\sum_{j=1}^m \absinner{y_j}{w}}{\max_i \absinner{g(x_i)}{w}}.
    \end{equation}
    We know from \eqref{eq:def_of_a} that 
    \begin{equation*}
        \absinner{y_j}{w} \leq \sum_{\l \leq j} a_{j\l}\absinner{g(x_\l)}{w} \leq m \, \max_{j,\l} \abs{a_{j\l}} \; \max_i \absinner{g(x_i)}{w},
    \end{equation*}
    which, when combined with inequalities \eqref{eq:sin_phi_over_sin_theta_bound} and \eqref{eq:a_pq_bound}, tells us that
    \begin{equation*}
        \frac{\sin(\phi)}{\sin(\theta)} \leq 4rm^2 \p[\big]{4mrD^{-1}}^m,
    \end{equation*}
    which is the constant bound we require.
    We may recall that $m \leq d$ and substitute in the lower bound for $D$ found in \eqref{eq:D_bound} to conclude that we may take
    \begin{equation*}
        \eta \leq 4d^2r\p[\big]{4dr(8dr)^d}^d \leq (8dr)^{3d^2}.
    \end{equation*}
    This completes the proof of \Cref{lem:angle-relation-bound}.
\end{proof}

It remains to compute the bound on the constant $C$ in \Cref{thm:clique-factor}.

In the proof of \Cref{thm:clique-factor} it was necessary only that $\norm{w}$ satisfies the assumption in \Cref{obs:angle-jump} and that $\norm{w} > 2\alpha / n$. Thus, to obtain a value of $C$, we need only that if $\norm{w} \geq C$ then $\norm{w}$ satisfies the assumption in \Cref{obs:angle-jump}. Thus we can take

\[C = \frac{\beta_0 \eta^{|X|-1}}{\beta_1},\]
for $\beta_0$, $\beta_1$, and $\eta$ as in the above three lemmas.

Since we have $|X| < (4r)^k = (4r)^{d+1}$, we can take

\begin{align*}
    C &= 8(r-1)^2 \p[\Big]{(8dr)^{3d^2}}^{(4r)^{d+1}} (8dr)^{d} \\
    &\leq (8dr)^{2 + 3d^2(4r)^{d+1} + d} \\
    &\leq (8dr)^{(8dr)^{d+1} - 1},
\end{align*}
as claimed in \Cref{thm:clique-factor}.


\section{Moving from a clique-factor to an \texorpdfstring{$H$}{H}-factor}
\label{sec:h-factor}

In this section, we deduce \Cref{thm:main} from \Cref{thm:clique-factor}.
We know from \Cref{thm:clique-factor} that there is a constant $C_{K_r}$ depending on $d,k,r,S$ such that, for every integer $n$ and colouring $c\from E(K_{nr}) \to S$ such that $\norm{\sum_{e\in E(K_{nr})} c(e)} \leq \alpha$, there is a $K_r$-factor $F'$ of $K_{nr}$ such that
\begin{equation*}
    \norm[\bigg]{\sum_{e\in E(F')} c(e) \; } \leq C_{K_r} + \frac{2\alpha}{n}.
\end{equation*}

Given such a $K_r$-factor $F'$, let $J_1,\dotsc,J_n$ be the copies of $K_r$ in $F'$.
Each $J_i$ inherits a colouring $c_i\from E(K_r) \to S$ induced by $c$.
As $\card{S} = k$, there are at most $k^{\binom{r}{2}}$ distinct colourings amongst $c_1,\dotsc,c_n$.

We now find an $H$-factor that is a subgraph of $F'$, recalling that $\card{V(H)} = r$ and so we can place one copy of $H$ inside each $K_r$.
Roughly, amongst all those copies of $K_r$ given the same colouring by $c$, we want to cover each edge by a copy of $H$ an equal number of times.
We do this approximately, and then account for the error term, which we can show is constant.

For a given $c'\from E(K_r) \to S$, let $X_{c'} \sseq [n]$ be the set of those $i$ such that $c_i = c'$.
Arbitrarily partition each $X_i$ as
\[X_{c'} = Y_{c',1} \union \dotsb \union Y_{c',h_{c'}} \union Z_{c'},\]
where $\card{Y_{c',1}} = \dotsb = \card{Y_{c',h_{c'}}} = r!$, and $\card{Z_{c'}} < r!$ (so $h_{c'} = \floor{\card{X_{c'}} / r!})$.

For each $Y_{c',j}$, let $f \from Y_{c',j} \to S_r$ be a bijection to the set of permutations of $[r]$, which of course has $r!$ elements.
Then for each element $t\in Y_{c',j}$, embed a copy of $H$ in $J_t$, embedding vertices at locations determined from the permutation $f(t)$.

For the copies of $K_r$ in $Z_{c'}$, we choose embeddings of $H$ arbitrarily. 
As $\card{Z_{c'}} < r!$, this will only contribute a constant error term.
Let this $H$-factor of $X_{c'}$ be $F_{c'}$. 
For any individual copy of $H$ embedded in $J_i$ (a copy of $K_r$), we have
\begin{equation*}
    \norm[\bigg]{\sum_{e\in E(H)} c(e) - \frac{e(H)}{e(K_r)}\sum_{e\in E(J_i)} c(e) \; } \leq 2\binom{r}{2}.
\end{equation*}
Thus, noting that the two sums cancel out when summed over a set $Y_{c',j}$ in the above embeddings, we have
\begin{equation*}
    \norm[\bigg]{\sum_{e\in E(F_{c'})} c(e) - \frac{e(H)}{e(K_r)}\sum_{i\in X_{c'}}\sum_{e\in E(J_i)} c(e) \; } \leq 2r!\binom{r}{2}.
\end{equation*}
Let $F$ be the $H$-factor of $K_nr$ consisting of the union of all $F_{c'}$s. Then, summing over all classes, we have
\begin{equation*}
    \norm[\bigg]{\sum_{e\in E(F)} c(e) - \frac{e(H)}{e(K_r)}\sum_{e\in E(F')} c(e) \; } \leq 2k^{\binom{r}{2}}r!\binom{r}{2}.
\end{equation*}
Thus we have
\begin{align*}
    \norm[\bigg]{\sum_{e\in E(F)} c(e) \; } 
    &\leq \frac{e(H)}{e(K_r)} \left( C_{K_r}+ \frac{2\alpha}{n} \right) + 2k^{\binom{r}{2}}r!\binom{r}{2} \\
    &= \left( \frac{e(H)}{e(K_r)} C_{K_r} + 2k^{\binom{r}{2}}r!\binom{r}{2} \right) + \frac{2 \alpha}{n},
\end{align*}
and so
\begin{equation*}
    C = \frac{e(H)}{e(K_r)} C_{K_r} + 2k^{\binom{r}{2}}r!\binom{r}{2}
\end{equation*}
will do for a graph $H$ in \Cref{thm:main}.

Finally, in the case when $S$ is a regular simplex, we may substitute in the value of $C_{K_r}$ from \Cref{thm:clique-factor} and replace $k$ with $d+1$ to find that
\begin{align*}
    C \leq C_{K_r} + 2(d+1)^{r^2}r^{r+2} \leq (8dr)^{(8dr)^{d+1}-1} + (8dr)^{r^2} \leq (8dr)^{(8dr)^{d+1}},
\end{align*}
as required, completing the proof of \Cref{thm:main}.
\qed


\section{Concluding remarks}
\label{sec:conclusion}

We have proved a uniform upper bound on the optimal balance of an $H$-factor in a $K_n$ equipped with a balanced $r$-colouring.

It seems likely that the constants we obtain in \Cref{thm:main-k-colouring} and \Cref{thm:main} are also far from optimal, and improving the dependencies here would be of some interest.
In particular, it is not immediately clear that $C$ must depend on both the number of colours and the size of the graph $H$, and very little is known about lower bounds for these problems.

There are also many other directions that remain to be understood, and many potential avenues for improving our main results.
Indeed, one may note that, when there are only two colours, significantly more is known. 
For example, it is shown in \cite{HLMP24} that one can find colour-balanced copies of spanning graphs $H$ with only a bound on the maximum degree of $H$, and one can even take the host graph $G$ to be merely a graph of sufficiently large maximum degree rather than a complete graph.
With this in mind, we ask the following question, which, if answered affirmatively, would give a significant generalisation of \Cref{thm:main-k-colouring}.

\begin{question}
\label{q:general}
    For fixed integers $k,r,d,\Delta \geq 2$, is there a constant $C=C_{k,r,\Delta}$ such that the following holds?
    For every integer $n$, set $S\sseq \SS^{d-1} \sseq \RR^d$ of size $k$, colouring $c\from E(K_n) \to S$ such that $\norm{\sum_{e\in E(K_n)} c(e)} \leq 1$, and every graph $H$ on $n$ vertices with maximum degree at most $\Delta$, there is a bijection $f\from V(H)\to V(G)$ such that
    \begin{equation*}
        \norm[\bigg]{\sum_{e\in E(H)} c(f(e)) \; } \leq C,
    \end{equation*}
    where, if edge $e$ connects vertices $u$ and $v$, then $f(e)$ is understood to be the edge of $K_n$ between $f(u)$ and $f(v)$.
\end{question}

It would also be of interest to consider, in the context of \Cref{thm:main}, whether one is required to fix a finite set $S$ of colours in advance.

\begin{question}
\label{q:continuous}
    What upper bounds can be placed, for fixed integers $d,r$ and fixed $r$-vertex graph $H$, on $\norm[\big]{\sum_{e\in E(F)} c(e)}$ for $F$ an $H$-factor of a complete graph $K_{rn}$ with $c\from E(H_{rn}) \to \SS^{d-1}$?
    Can there be an upper bound independent of $n$?
\end{question}

There are of course many other potential avenues of investigation besides the above two questions.
For instance, one could ask \Cref{q:general} when $H$ is restricted to be a tree, or in \Cref{q:continuous}, one could ask for an upper bound on $\min_F \norm[\big]{\sum_{e\in E(F)} c(e)}$ over all $H$-factors $F$ in terms only of $H$ and the size of the image of $c$.


\section{Acknowledgements}

The first author is supported by the Engineering and Physical Sciences Research Council (EPSRC).
The second author is funded by the internal graduate studentship of Trinity College, Cambridge. The authors would like to thank their PhD supervisor, B\'{e}la Bollob\'{a}s, for his comments on an earlier version of the manuscript.


\bibliographystyle{abbrvnat}  
\renewcommand{\bibname}{Bibliography}
\bibliography{main}


\end{document}